\documentclass[a4paper,oneside,reqno,final]{amsart}
\usepackage[utf8]{inputenc}
\usepackage[a4paper,margin=3cm]{geometry} 
\usepackage{bm}
\usepackage{amsmath}
\usepackage{amsthm}
\usepackage{amscd}
\usepackage{MnSymbol}
\usepackage{dsfont}
\usepackage{mathtools}
\usepackage{enumitem}
\usepackage[notref,notcite]{showkeys}

\usepackage{graphicx}
\usepackage{float}
\usepackage{array,booktabs}

\usepackage[colorlinks,pdftex]{hyperref}
\hypersetup{
linkcolor=black,
citecolor=black,
pdftitle={},
pdfauthor={Franziska K\"uhn},
pdfkeywords={},
}

\renewcommand\labelenumi{\textup{(\roman{enumi})}}
\renewcommand\theenumi\labelenumi
\renewcommand\labelenumii{(\alph{enumii})}
\renewcommand\theenumii\labelenumii

\renewcommand\theenumii\labelenumii

\swapnumbers
\theoremstyle{theorem} \newtheorem{theorem}{Theorem}[section]
\theoremstyle{theorem} \newtheorem{lemma}[theorem]{Lemma}
\theoremstyle{theorem} 
\theoremstyle{theorem} 
\theoremstyle{definition} 
\theoremstyle{definition} 
\theoremstyle{definition}  
\theoremstyle{definition}  \newtheorem*{ack}{Acknowledgement}

\DeclareMathOperator \supp {supp}

\DeclareFontFamily{U}{mathx}{\hyphenchar\font45}
\DeclareFontShape{U}{mathx}{m}{n}{
      <5> <6> <7> <8> <9> <10>
      <10.95> <12> <14.4> <17.28> <20.74> <24.88>
      mathx10
      }{}
\DeclareSymbolFont{mathx}{U}{mathx}{m}{n}
\DeclareFontSubstitution{U}{mathx}{m}{n}
\DeclareMathAccent{\widecheck}{0}{mathx}{"71}
\DeclareMathAccent{\wideparen}{0}{mathx}{"75}

\newcommand{\I}{\mathds{1}}
\newcommand{\Ee}{\mathds{E}}

\newcommand{\nat}{\mathds{N}}
\newcommand{\real}{\mathds{R}}
\newcommand{\Scal}{\mathcal{S}}
\newcommand{\Fscr}{\mathcal{F}}
\newcommand\scalp[2]{\langle #1,\,#2\rangle}
\newcommand\lrscalp[2]{\left\langle #1,\,#2\right\rangle}

\newcommand\mc[1] {\mathcal{#1}}
\newcommand\mbb[1] {\mathds{#1}}

\makeatletter
\@namedef{subjclassname@2020}{%
  \textup{2020} Mathematics Subject Classification}
\makeatother

\begin{document}

\title{For which functions are $f(X_t)-\Ee f(X_t)$ and $g(X_t)/\Ee g(X_t)$ martingales?}
\author[F.~K\"{u}hn]{Franziska K\"{u}hn}
\address[F.~K\"{u}hn]{TU Dresden, Fakult\"at Mathematik, Institut f\"{u}r Mathematische Stochastik, 01062 Dresden, Germany. E-Mail: \textnormal{franziska.kuehn1@tu-dresden.de}}

\author[R.L.~Schilling]{Ren\'e L.\ Schilling}
\address[R.L.~Schilling]{TU Dresden, Fakult\"at Mathematik, Institut f\"{u}r Mathematische Stochastik, 01062 Dresden, Germany. E-Mail: \textnormal{rene.schilling@tu-dresden.de}}

\subjclass[2020]{60G44; 60G51; 60J65; 39B22; 45E10}
\keywords{L\'evy process; Brownian motion; martingale; polynomial process; convolution equation; Choquet--Deny theorem; Cauchy functional equation; harmonic polynomial.}

\begin{abstract}
    Let $X=(X_t)_{t\geq 0}$ be a one-dimensional L\'evy process such that each $X_t$ has a $C^1_b$-density w.r.t.\ Lebesgue measure and certain polynomial or exponential moments. We characterize all polynomially bounded functions $f:\real\to\real$, and exponentially bounded functions $g:\real\to (0,\infty)$, such that $f(X_t)-\Ee f(X_t)$, resp.\ $g(X_t)/\Ee g(X_t)$, are martingales.
\end{abstract}

\dedicatory{\emph{Accepted for publication in} Theory of Probability and Mathematical Statistics}
\maketitle

In a series of papers \cite{man-tev20,man-tev20b,man-tev21} M.~Mania and R.~Tevzadze studied the interplay of certain Brownian martingales of the form $M_t = f(B_t) - \Ee f(B_t)$ and $N_t = g(B_t)/\Ee g(B_t)$, respectively, and Cauchy's functional equation; in particular, they characterized all functions $f$ and $g$ such that $M$ and $N$ are indeed martingales. Their approach is mainly analytic, based on the solution of a certain `backward' heat equation. In this short note, we provide several alternative, mainly probabilistic, proofs which (i) recover for Brownian motion the results by Mania and Tevzadze, and (ii) extend to the more general class of L\'evy processes.

Our main result can be summarized as follows. Let $(X_t)_{t\geq 0}$ be a one-dimensional L\'evy process with a strictly positive transition density:
\begin{itemize}
\item if the transition density is of class $C^1_b$, if $f$ is polynomially bounded, and if $X_t$ has certain polynomial moments such that $\Ee f(X_t)$ exists, then the process $M_t=f(X_t)-\Ee f(X_t)$ is a martingale if, and only if, $f$ is a.e.\ a second-order polynomial (Theorems~\ref{main-01} and~\ref{main-05}).
\item if $g$ is a positive function which grows at most exponentially, and if $X_t$ has exponential moments such that $\Ee g(X_t)$ exists, then the process $N_t = g(X_t)/\Ee g(X_t)$ is a martingale if, and only if, $g(x) = ae^{\lambda_1 x}+be^{\lambda_2 x}$ a.e.\ for suitable coefficients $\lambda_1,\lambda_2\in \real$ and $a,b\geq 0$. This still holds, if $X_t$ is a one-sided L\'evy process (Theorems~\ref{main-11} and~\ref{main-13})
\end{itemize}
Note that all of these conditions trivially hold for Brownian motion, i.e.\ we recover the results by Mania and Tevzadze.

\section{L\'evy processes}
A stochastic process $X=(X_t)_{t\geq 0}$ with values in $\real$ is a \emph{L\'evy process}, if $t\mapsto X_t$ is right-continuous with finite left-hand limits (c\`adl\`ag), and if its increments are independent and stationary random variables, i.e.\
\begin{gather}\label{intro-e03}
    (X_{t_k}-X_{t_{k-1}})_{k=1,\dots, n} \quad\text{are independent and}\quad X_{t_k}-X_{t_{k-1}}\sim X_{t_k-t_{k-1}}
\end{gather}
for any choice of $n\in\nat$ and $t_0=0\leq t_1\leq \dots\leq t_n<\infty$. Note that this entails $X_0\sim\delta_0$, i.e.\ $X_0=0$ a.e. If $\Fscr_t := \sigma(X_s, s\leq t)$ denotes the natural filtration, we can express \eqref{intro-e03} equivalently in the following form:
\begin{gather}\label{intro-e05}
    \Ee\left(e^{i\xi (X_{t+s}-X_s)} \mid \Fscr_s\right) = e^{-t\psi(\xi)},\quad s,t\geq 0,\;\xi\in\real,
\end{gather}
where $\psi(\xi)$ is the \emph{characteristic exponent} which is given by the L\'evy--Khintchine representation
\begin{gather}\label{intro-e07}
    \psi(\xi) = -ib\xi + \frac 12\sigma^2\xi^2 + \int_{y\neq 0} \left(1-e^{iy\xi}+iy\xi\I_{(0,1)}(|y|)\right)\nu(dy).
\end{gather}
The exponent $\psi$ determines (the finite dimensional distributions of) $X$ uniquely, and $\psi$ is uniquely determined by the \emph{L\'evy triplet} $(b,\sigma^2,\nu)$ comprising the drift $b\in\real$, the diffusion coefficient $\sigma^2\geq 0$ and the L\'evy measure $\nu$, i.e.\ a measure on $\real\setminus\{0\}$ satisfying $\int_{y\neq 0} \min\{y^2,1\}\,\nu(dy)<\infty$. The standard reference for L\'evy processes is Sato \cite{sato}, an introduction is given in \cite{barca}.

Any L\'evy process is a Markov process and the transition semigroup is given for any bounded Borel function $u:\real\to\real$ by
\begin{gather}\label{intro-e09}
    T_t u(x) = \Ee^x u(X_t) = \Ee u(X_t+x);
\end{gather}
notice the convolution structure of $T_t u = \widetilde\mu_t* u$, where $\widetilde\mu_t$ is the law of $-X_t$, which is due to the independent and stationary increment property \eqref{intro-e03}. If $u$ is smooth enough, say if $u$ is in the L.\ Schwartz space of rapidly decreasing, $C^\infty$-functions $u\in\Scal(\real)$, then we can use the Fourier transform to determine the infinitesimal generator of $(X_t)_{t\geq 0}$ or $(T_t)_{t\geq 0}$ as follows: write $\widehat u(\xi) = (2\pi)^{-1} \int e^{-ix\xi} u(x)\,dx$ and $\widecheck u(x) = \int e^{i\xi x} u(\xi)\,d\xi$ for the Fourier and the inverse Fourier transform, respectively. Then
\begin{gather}
    \widehat{Au}
    = \lim_{t\to 0}\frac 1t\left(\widehat{T_tu} - \widehat{u}\right)
    = \lim_{t\to 0}\frac 1t\left(e^{-t\psi}\widehat u - \widehat{u}\right)
    = -\psi\cdot\widehat u,
\end{gather}
i.e.\ $Au = -\psi(D)u = (-\psi\cdot\widehat{u})\widecheck{\phantom{u}}$ is a pseudo-differential operator with symbol $-\psi$. Combining this with the L\'evy--Khintchine formula we easily see that
\begin{gather}\label{intro-e11}
    Au(x)
    = bu'(x) + \frac 12 \sigma^2u''(x) + \int_{y\neq 0}\left[u(x+y)-u(x)-yu'(x)\I_{(0,1)}(|y|)\right]\nu(dy).
\end{gather}
If we consider L\'evy processes and their semigroups in the Banach space $(C_\infty(\real),\|\cdot\|_\infty)$ of continuous functions vanishing at infinity, which is naturally equipped with the uniform convergence, then the (strong) generator $A$ has the domain
\begin{align*}
    \mathcal D(A)
    &= \left\{u\in C_\infty(\real); \|\cdot\|_\infty\text{-}\lim_{t\to0}\tfrac 1t \left(T_tu-u\right)\text{\ \ exists}\right\}\\
    &= \left\{u\in C_\infty(\real); \lim_{t\to 0}\tfrac 1t \left(T_tu(x)-u(x)\right)=g(x)\text{\ \ and\ \ }g\in C_\infty(\real)\right\}.
\end{align*}

We can use this representation to extend $A$ to twice differentiable functions $u\in C_b^2(\real)$, and beyond. Assume that $\int_{|y|\geq 1} |y|^n\,\nu(dy)<\infty$, then it follows from Taylor's formula and \eqref{intro-e11} that $Au(x)$ exists if $u''(x)$ grows at most like $|x|^{n-2}$ as $|x|\to\infty$. In fact, it is possible to show that for $g(x)=|x|^n$
\begin{gather}\label{intro-e13}
    \forall t\::\:\Ee g(X_t) < \infty
    \iff \int_{|y|\geq 1} g(y)\,\nu(dy) < \infty
    \iff \forall \phi \in C_c^\infty(\real)\::\: A^*\phi\in L^1(g(x)\,dx),
\end{gather}
cf.\ e.g.\ \cite[Theorem~3]{ber-kue-schi}. Here, $A^*$ is the (formal) adjoint of the operator $A$. Since $\widehat{A\phi} = -\psi\widehat{\phi}$, it is not hard to see that $A^*$ is the generator of the L\'evy process $(-X_t)_{t\geq 0}$ and that it is a pseudo-differential operator with the conjugate-complex symbol $-\overline{\psi(\xi)}$. Using \eqref{intro-e13} we can extend, in a weak sense, the generator $A$ to the space of polynomially bounded Borel functions $B(\real)$ via
\begin{gather*}
    \scalp{Af}{\phi} := \scalp{f}{A^*\phi},\quad\phi\in C_c^\infty(\real).
\end{gather*}
We will call this the \emph{weakly extended generator}.  Write $g(x)=\text{bp-}\lim_{t\to 0} \frac 1t(T_t f(x) - f(x))$ if the limit exists pointwise for each $x$ and $\sup_x \frac 1t |T_tf(x)-f(x)|  < \infty$ (`bp' stands for `boundedly pointwise'). If $f\in B(\real)$ is polynomially bounded, then $g=Af$ for the weakly extended generator. Indeed, denote by $(T_t^*)_{t \geq 0}$ the semigroup associated with $(-X_t)_{t \geq 0}$. Applying Dynkin's formula, and then Fubini's theorem, yields that
\begin{align} \label{intro-e15} \begin{aligned}
	\lrscalp{\frac{T_tf-f}{t}}{\phi}
	= \lrscalp{f}{\frac{T_t^* \phi-\phi}{t}}
	= \frac{1}{t} \int_0^t \lrscalp{f}{T_s^* A^* \phi} \, ds
	&= \frac{1}{t} \int_0^t \lrscalp{T_s f}{A^* \phi} \, ds \\
	&= \int_\real \frac{1}{t} \int_0^t \mbb{E}f(x+X_s) \, ds \, A^* \phi(x) \, dx.
	\end{aligned}
\end{align}
Note that Fubini's theorem is applicable because of \eqref{intro-e13} and the polynomial boundedness of $f$; the latter implies that
\begin{align*}
	\left| \frac{1}{t} \int_0^t \mbb{E}f(x+X_s) \, ds \right|
	\leq \sup_{s \leq t} |\mbb{E}f(x+X_s)|
	\leq C (1+|x|)^n \sup_{s \leq t} \mbb{E}((1+|X_s|)^n) < \infty,
\end{align*}
see  e.g.\  \cite[Theorem~3]{ber-kue-schi}. Because of the existence of the bp-limit, the mapping $s \mapsto \mbb{E}f(x+X_s) = T_s f(x)$ is continuous at $s=0$, and so
\begin{equation*}
	\lim_{t \to 0} \frac{1}{t} \int_0^t \mbb{E}f(x+X_s) \, ds = f(x).
\end{equation*}
Using the previous estimate and \eqref{intro-e13}, we can let $t \to 0$ in \eqref{intro-e15} using dominated convergence and get $\scalp{g}{\phi} = \scalp{f}{A^* \phi}$.

\section{Some auxiliary results}\label{sec-pre}

In this section we collect a few results which will be used in the proof of the main results.
\begin{lemma}\label{pre-03}
    Let $(X_t)_{t\geq 0}$ be a one-dimensional L\'evy process with characteristic exponent $\psi$ such that $\Ee\left[|X_t|^{n}\right]<\infty$ for some $n\in\nat$. If $p:\real\to\real$ is a polynomial of degree $\deg(p)=n$, then $q(t):=\Ee\left[p(X_t)\right]$ is a polynomial of degree $\deg(q)\leq n$.
\end{lemma}
\begin{proof}
   It is enough to show the assertion for $p(x)=x^n$. Since $X_t$ has $n$th moments, we can use the characteristic function to see
    \begin{align*}
        q(t)
        = \Ee\left[X_t^n\right]
        = \frac 1{i^n} \partial_\xi^n \Ee\left[e^{i\xi X_t}\right]\Big|_{\xi = 0}
        = \frac 1{i^n} \partial_\xi^n e^{-t\psi(\xi)}\Big|_{\xi = 0}.
    \end{align*}
    An application of Faa di Bruno's formula yields
    \begin{align*}
        \partial_\xi^n e^{-t\psi(\xi)}
        = \sum_{k=1}^n (-t)^k e^{-t\psi(\xi)} \underbracket[.6pt]{\sum_{i_1=0}^k \ldots\ldots\sum_{i_n=0}^k}_{\scriptsize\begin{gathered}i_1+i_2+\dots+i_n=k\\[-4pt] i_1+2i_2+\dots+ni_n=n\end{gathered}} \frac{n!}{i_1!\cdot i_2! \cdots i_n!} \prod_{j=1}^n \left(\frac{\partial^j\psi(\xi)}{j!}\right)^{i_j}.
    \end{align*}
    Since $\psi(0)=0$, we conclude that $\deg(q)\leq n$.
\end{proof}

The key technique in the proof of the main result is the use of Liouville-type results for the generator of certain L\'evy processes. The following theorem is due to one of us \cite[Thm.~1]{franzi-liouville}. We state here only the one-dimensional case, but it holds in any dimension.
\begin{theorem}[Liouville]\label{pre-05}
    Let $(X_t)_{t\geq 0}$ be a one-dimensional L\'evy process such that each $X_t$ has a transition density $p_t\in C_b^1(\real)$ and $\Ee\left[|X_t|^{n+\epsilon}\right]<\infty$ for some $n\in\nat$ and $\epsilon>0$. Denote by $A$ the infinitesimal generator of $(X_t)_{t \geq 0}$. If $u:\real\to\real$ is polynomially bounded such that $|u(x)|\leq c (1+|x|^{n})$ and a weak solution to the equation $Au=0$, then $u$ is a polynomial of degree $\deg(u)\leq n$.
\end{theorem}

If the characteristic exponent $\psi$ of the L\'evy process $(X_t)_{t \geq 0}$ is smooth, then a  standard result shows $\Ee\left[X_t^{2k}\right]<\infty$ for all $k \in \nat$, cf.\ \cite[\S2.12, Thm.~1, p.~334]{shir} or \cite[Lemma~4]{ber-schi}, and so $u$ can be of arbitrary polynomial growth. We give a slightly different version of Liouville's theorem and present an analytical proof. Recall that $\Scal(\real)$ are the rapidly decreasing, smooth L.~Schwartz functions and their dual space $\Scal'(\real)$ are the tempered distributions.
\begin{lemma}\label{pre-07}
    Let $(X_t)_{t\geq 0}$ be a one-dimensional L\'evy process with characteristic exponent $\psi$. Moreover, assume that $\psi\in C^\infty(\real)$ is smooth and $\xi=0$ is the only zero of $\psi(\xi)$. Then $\Ee\left[|X_t|^n\right]<\infty$ for all $n\in\nat$, and any polynomially bounded measurable function $u:\real\to\real$ satisfying $Au = c$ weakly is a polynomial.
\end{lemma}

It is a classic result that $X_t$ has a [possibly degenerate] lattice distribution if, and only if, $\psi$ has a zero $\xi_0\neq 0$. In this case, $\psi$ is a [degenerate, i.e.\ constant] periodic function, cf.\ Sato \cite[Section 24]{sato} or \cite[Theorem~10]{ber-kue-schi}. The following proof works almost literally for any dimension $d\geq 1$, but we need only the case $d=1$.

\begin{proof}
    By assumption, the characteristic function $\Ee e^{i\xi X_t} = e^{-t\psi(\xi)}$ is of class $C^{2k}(\real)$ for any $k\in\nat$, and so $\Ee\left[X_t^{2k}\right]<\infty$. It is not hard to see with the help of the L\'evy--Khintchine formula, cf.\ \cite[Lemma~4]{ber-schi}, that a smooth L\'evy exponent $\psi$ and all of its derivatives $\partial^n\psi$ are polynomially bounded. In particular, $\overline{\psi}\cdot\phi$ is for every Schwartz function $\phi\in\Scal(\real)$ again a Schwartz function.

    Since $u$ is polynomially bounded, the inverse Fourier transform $\widecheck u$ is well-defined as a tempered distribution. For any $\phi\in\Scal(\real)$ we have
    \begin{gather*}
        \scalp{Au}{\phi}
        = \scalp{u}{A^*\phi}
        = -\scalp{u}{\overline{\psi}(D)\phi}
        = -\scalp{\widecheck u}{\overline\psi\widehat\phi}.
    \end{gather*}
    In the last equality we use Plancherel's theorem for the canonical dual pairing of $\Scal'(\real)$ and $\Scal(\mbb{R})$; here we need that $\overline\psi\cdot\phi\in\Scal(\real)$. On the other hand, $Au = c$ weakly gives
    \begin{gather*}
        -\scalp{\widecheck u}{\overline\psi\widehat\phi}
        = \scalp{Au}{\phi}
        = \int_\real c \phi(x)\,dx
        = 2\pi c \widehat\phi(0)
        = \scalp{2\pi c\delta_0}{\widehat\phi}.
    \end{gather*}
    Since this holds for all $\phi\in\Scal(\real)$, and since the Fourier transform is a bijection on $\Scal(\real)$, we conclude from this that $\widecheck u \overline\psi = - 2\pi c \delta_0$. Now we can use the fact that $\xi=0$ is the only zero of $\psi$. Therefore, $\widecheck u$ has support $\supp u\subset\{0\}$, and a standard structural result for distributions supported in $\{0\}$, see e.g.\ Rudin \cite[Thm.~6.25]{rudin-fa}, shows that
    \begin{gather*}
        \widecheck u = \sum_{k=0}^N c_k\partial^k\delta_0
    \end{gather*}
    for some $N\in\nat$. From this we get by Fourier inversion that $u(x) = (2\pi)^{-1}\sum_{k=0}^N c_k (-i)^k x^k$.
\end{proof}

The following theorem shows that the polynomial $u$ appearing in Theorem~\ref{pre-05} or Lemma~\ref{pre-07} has at most degree $2$.
\begin{lemma}\label{pre-09}
    Let $(X_t)_{t\geq 0}$ be a one-dimensional L\'evy process such that $\Ee\left[|X_t|^{n_0}\right]<\infty$ for some $n_0\in\nat$. Assume that $(X_t)_{t \geq 0}$ is non-trivial, i.e.\ $X_t \not\equiv 0$. If $p$ is a polynomial of degree $\deg(p)\leq n_0$ such that $Ap \equiv\text{const.}$, then $\deg(p)\leq 2$.
\end{lemma}
\begin{proof}
    Let $(b,\sigma^2,\nu)$ be the L\'evy triplet of the characteristic exponent $\psi$. Since $\Ee\left[|X_t|^{n_0}\right]<\infty$ entails that $\int_{|y|\geq 1} |y|^{n_0}\,\nu(dy)<\infty$, see e.g.\ \cite[Thm.~3]{ber-kue-schi}, we can use the integro-differential form of the generator
    \begin{gather*}
        Au(x)
        = bu'(x) + \frac 12 \sigma^2u''(x) + \int_{y\neq 0}\left[u(x+y)-u(x)-yu'(x)\I_{(0,1)}(|y|)\right]\nu(dy)
    \end{gather*}
    to extend $A$ to all polynomials of degree less or equal than $n_0$. For $n\leq n_0$ we get
    \begin{align*}
        Ax^n
        &= b nx^{n-1} + \frac 12\sigma^2 n(n-1)x^{n-2}
        + \sum_{k=2}^n \binom nk x^{n-k}\int_{|y|<1} y^k\,\nu(dy)
        + \sum_{k=1}^n \binom nk x^{n-k}\int_{|y|\geq 1} y^k\,\nu(dy)\\
        &= (b+M_1) nx^{n-1} + \binom n2\left(\sigma^2+m_2+M_2 \right) x^{n-2}
        + \sum_{k=3}^n \binom nk \left(m_k+M_k\right) x^{n-k},
    \end{align*}
    where we use the shorthand notation $m_k = \int_{|y|<1} y^k\,\nu(dy)$ and $M_k = \int_{|y|\geq 1} y^k\,\nu(dy)$.

   Now let $p(x) = a_n x^n + \dots + a_1x + a_0$ be a polynomial such that $Ap(x)\equiv c$. Without loss of generality we may assume that $a_n=1$. Then
   \begin{equation*}
	    Ap(x) = (b+M_1) n x^{n-1} + \binom{n}{2} (\sigma^2+m_2+M_2) x^{n-2} + \sum_{k=0}^{n-3} c_k x^k
	 \end{equation*}
    for some coefficients $c_k \in \real$. If $\deg(p)=n \geq 3$, then it follows from  $Ap(x) \equiv c$ by a comparison of coefficients that $\sigma^2 + m_2+M_2 =0$, implying $\sigma^2=0$ and $\nu=0$, and $0=b+M_1 =b$. Hence, $X_t \equiv 0$ is the trivial L\'evy process. In other words, if $(X_t)_{t \geq 0}$ is non-trivial, then any polynomial $p$ with $Ap(x) \equiv c$ satisfies $\deg(p) \leq 2$.
\end{proof}

Denote by $\Delta_y f(x) := f(x+y)-f(x)$ the difference with step $y$ and by $\Delta^{n+1}_yf := \Delta_y(\Delta_y^n f)$ the iterated differences. The next lemma can be reduced to the form $\Delta_y^n f=0$ and then be derived from what is known in the literature as \emph{Fr\'echet's functional equation} or \emph{Cauchy's generalized functional equation}, cf.\ \cite[pp.~129--130]{aczel} for an overview on the literature (but without proofs). For continuous solutions $f$, the earliest proof is due to Anghelutza \cite{angel}, a proof via approximation can be found in Butzer \& Kozakiewicz \cite{butzer}. The following short and elementary argument seems to be new.
\begin{lemma}\label{pre-10}
    Let $\mathcal P_n$ be the family of polynomials of degree $n$. Any continuous solution of the problem
    \begin{gather}\label{pre-e31}
        f(x+y) - f(x) \in \mathcal P_n,\quad y\in\real,
    \end{gather}
    is a polynomial of degree $n+1$.
\end{lemma}
\begin{proof}
    Fix $y\neq 0$, assume that $f(x+y)-f(x) = p_y(x)$ where $p_y\in\mathcal P_n$; clearly, the coefficients of $p_y$ may depend on the parameter $y$. It is not hard to see that this problem is solved by some polynomial of degree $n+1$. Indeed, if we expand $p_y(x)$ into a Newton interpolation series,
    \begin{gather*}
        p_y(x) - p_y(0) = \sum_{k=1}^n a_k(y) x^{\left(\frac ky\right)},\quad
        x^{\left(\frac ky\right)} := x\cdot(x-y)\cdot\ldots\cdot (x-(k-1)y),
    \end{gather*}
    we can reduce the problem to finding polynomials $q_k$ of degree $k+1$ such that $\Delta_y q_k(x) = x^{\left(\frac ky\right)}$ for $k=1,\dots, n$. Obviously, $\Delta_y x^{\left(\frac {k+1}{y}\right)} = (k+1)x^{\left(\frac ky\right)}$, and so $q_y(x) = p_y(0) + \sum_{k=1}^n  a_k(y)  q_k(x)$ is a polynomial solving the original problem.

    If $F_1$ and $F_2$ are solutions to $f(x+y)-f(x) = p_y(x)$, then
    \begin{gather*}
        \Delta_y(F_2-F_1) = \Delta_y F_2 - \Delta_y F_1 = 0,
    \end{gather*}
    which shows that $(F_2-F_1)(x) = (F_2-F_1)(x+y)$, i.e.\ the difference $F_2-F_1$ is a $y$-periodic function. Therefore, the general solution of $\Delta_y f = p_y$ is given by
    \begin{gather*}
        f(x) = q_y(x) + \pi_y(x)
    \end{gather*}
    where $q_y$ is the particular polynomial solution and $x\mapsto\pi_y(x)$ is a continuous $y$-periodic function.

    By assumption $\Delta_z f(x) = f(x+z)-f(x)\in\mathcal P_n$ for any $z\neq y$. We pick $z$ in such a way that $z/y\notin\mathds{Q}$. Using the representation of $f$ and the periodicity of $\pi_y$ we conclude that $\pi_y(x+z)-\pi_y(x)\equiv c$.

    This leads to the problem $\Delta_z\pi_y(x)=c$, and the discussion above shows that its general solution is of the form $\pi_y(x) = a(y,z)x + b(y,z) + \bar\pi_{y,z}(x)$ with a $z$-periodic continuous function $\bar\pi_{y,z}(x)$. Since $\pi_y$ is $y$-periodic, $a(y,z)=0$ and $\bar\pi_{y,z}$ is both $z$- and $y$-periodic. Because of our assumption $z/y\notin\mathds{Q}$, the periods of $\bar\pi_{y,z}(x)$ form a dense subset of $\real$; using the continuity we see that $\pi_{y,z}(x)$, hence $p_y(x)$, is constant. This proves $f = q_y +c\in\mathcal P_{n+1}$.
\end{proof}

Let us close this preparatory section with a simple but extremely useful result.
\begin{lemma}\label{pre-11}
    Let $X$ be a random variable such that $X \sim p(x)\,dx$ and $p>0$ Lebesgue a.e. If $u$ is a Borel measurable function such that $u(X)=c$ almost surely for some constant $c \in \mbb{R}$, then $u(x)=c$ Lebesgue almost everywhere.
\end{lemma}
\begin{proof}
	Clearly, $\Ee \left[(u(X)-c)^2\right]=0$. Thus,
	\begin{equation*}
		0 = \int_{\mbb{R}} (u(x)-c)^2 p(x) \, dx.
	\end{equation*}
	Since $p>0$ a.e., it follows that $u(x)=c$ Lebesgue almost everywhere.
\end{proof}

\section{Main results}

We are now ready to state and prove the main results of this note.
\begin{theorem}\label{main-01}
    Let $(B_t)_{t \geq 0}$ be a one-dimensional Brownian motion, and $f$ a Borel measurable function which is polynomially bounded. If $f(B_t)-\Ee\left[f(B_t)\right]$ is a martingale, then
    \begin{enumerate}
	\item\label{main-01-i}
        $f=\tilde{f}$ almost everywhere for a twice differentiable function $\tilde{f}$ satisfying $A \tilde{f}=\mathrm{const}$, where $A$ is the \textup{(}weakly extended\textup{)} generator,\footnote{$A$ is, of course, the Laplace operator $\frac 12\Delta$. In view of the extension to more general L\'evy processes, we prefer to use the notation $A$ here.}
	\item\label{main-01-ii}
        $f(x) = ax^2+bx+c$ almost everywhere for suitable constants $a,b,c \in \mbb{R}$.
	\end{enumerate}
\end{theorem}
\begin{proof}
    Set $\gamma(t) := \Ee f(B_t)$ and denote by $T_t f(x) := \Ee^x\left[f(B_t)\right] = \Ee\left[f(x+B_t)\right]$ the semigroup generated by the process $(B_t)_{t \geq 0}$. Without loss of generality, we may assume that $f(0)=0$, hence $\gamma(0)=0$. We are first going to show that $t \mapsto \gamma(t)$ is a linear function. Since $f(B_t)-\Ee\left[f(B_t)\right]$ is a martingale, we have
	\begin{equation*}
		\Ee\left[f(B_{t+s})-\gamma(t+s) \mid \mc{F}_s\right] = f(B_s)-\gamma(s)
	\end{equation*}
	for all $s,t\geq 0$, and so, by the Markov property,
    \begin{align*}
        f(B_s)-\gamma(s)
        = \Ee\left[f(B_{t+s})-\gamma(t+s)\right \mid \mc{F}_s]
        = \Ee^{B_s}\left[f(B_t)\right] - \gamma(t+s)
        = (T_tf)(B_s) - \gamma(t+s).
	\end{align*}
	We can rearrange this to get
	\begin{equation*}
		(T_{t} f-f)(B_s) = \gamma(t+s)-\gamma(s) \quad \text{almost surely}.
	\end{equation*}
    Since the transition density $p_s$ of $B_s$ is strictly positive, this implies that
	\begin{equation}\label{main-e04}
		(T_{t} f-f)(x)=\gamma(t+s)-\gamma(s) \quad \text{Lebesgue almost everywhere},
	\end{equation}
	see Lemma~\ref{pre-11}. If we choose $s=0$, then \eqref{main-e04} yields
	\begin{equation*}
		(T_{t} f-f)(x) = \gamma(t)\quad\text{Lebesgue almost everywhere},
	\end{equation*}
    and we get that
    \begin{equation*}
		\gamma(t+s) = \gamma(t)+\gamma(s), \quad s,t\geq 0,
	\end{equation*}
    which means that $\gamma$ solves the Cauchy functional equation. As $f$ is polynomially bounded, Lemma~\ref{pre-03} shows that the function $\gamma$ is bounded on every finite interval. Any solution to the Cauchy functional equation which is bounded on at least one non-degenerate interval is, however, linear, see e.g.\ Acz\'el \cite[Theorem~1, p.~34]{aczel}, and so $\gamma(t)=\alpha t+\beta$ for some constants $\alpha,\beta \in \mbb{R}$. From $\gamma(0)=0$, we conclude that $\beta=0$.

    Inserting this into \eqref{main-e04} for $s=0$, we see that
	\begin{equation*}
		T_t f(x) - f(x) = \gamma(t) = \alpha t \quad \text{Lebesgue almost everywhere.}
	\end{equation*}
    Since $p_t(y) = (2\pi t)^{-1/2}\exp\left[-\frac 1{2t} y^2\right]$ is a smooth kernel, the convolution $x\mapsto T_t f(x)$ is twice continuously differentiable, and so $f=\tilde{f}$ almost everywhere for a twice continuously differentiable function $\tilde{f}$. Because of the continuity, we then have $T_t \tilde{f}(x) - \tilde{f}(x)=\alpha t$ for all $x \in \mbb{R}$. In particular, for all $x\in\real$
	\begin{equation*}
		A\tilde{f}(x) = \lim_{t\to 0} \frac{T_t \tilde{f}(x)-\tilde{f}(x)}{t} = \alpha,
	\end{equation*}
    and so $A\tilde{f}=\text{const.}$, in the weak sense. We can now apply Lemma~\ref{pre-07} and~\ref{pre-09} and see that $\tilde{f}$ is a polynomial of degree $\deg(\tilde f)\leq 2$. This finishes the proof.
\end{proof}	

Let us add a further proof which uses only martingale arguments and avoids Cauchy's functional equation. We will present a further proof avoiding Cauchy's functional equation in Theorem~\ref{main-05} below, which generalizes Theorem~\ref{main-01} to L\'evy processes.

\begin{proof}[Alternative proof of Theorem~\ref{main-01}]
Assume that $f$ is a polynomially bounded function such that $f(B_t)-\Ee f(B_t)$ is a martingale.

We will first show that we can replace $f$ by a $C^2$-function. Pick $s,t\geq 0$ and fix $u> 0$. Set $f_u(x) := \Ee^x f(B_u) = \Ee f(B_u+x)$. Using the Markov property and the martingale property, we see
\begin{align*}
    \Ee\left[ f_u(B_{t+s}) - \Ee f_u(B_{t+s}) \mid \Fscr_t\right]
    &= \Ee\left[ \Ee^{B_{t+s}} f(B_u) - \Ee \Ee^{B_{t+s}} f(B_{u}) \mid \Fscr_t\right]\\
    &= \Ee\left[ f(B_{t+s+u}) - \Ee f(B_{t+s+u}) \mid \Fscr_t\right]\\
    &= f(B_t)-\Ee f(B_{t}).
\end{align*}
If we take $s=0$, the above calculation shows
\begin{align*}
    \Ee\left[ f_u(B_{t}) - \Ee f_u(B_{t}) \mid \Fscr_t\right]
    = f(B_t) - \Ee f(B_t)
    = \Ee\left[ f_u(B_{t+s}) - \Ee f_u(B_{t+s}) \mid \Fscr_t\right] ,
\end{align*}
and since $f_u(B_t)$ is $\Fscr_t$ measurable, we conclude that $M_t := f_u(B_t)-\Ee f_u(B_t)$ is a martingale.

Noting that $f_u(x) = \Ee f(B_u+x) = \int_\real f(y)p_u(y-x)\,dy$ with the smooth transition density $p_u(y)$ of $B_u$, we can use dominated convergence to see that $x\mapsto f_u(x)$ is a $C^2$-function. Thus, we can apply It\^o's formula to infer that $N_t := f_u(B_t) - \int_0^t Af_u(B_r)\,dr$ is a local martingale.

Let $(\sigma_n)_{n\in\nat}$, $\sigma_n\uparrow\infty$, be a localizing sequence for $(N_t)_{t\geq 0}$. Since $A f_u(x)$ is continuous, the stopping times
\begin{gather*}
    \tau_n := \sigma_n \wedge \inf\left\{t\geq 0 : |Af_u(B_t)|\leq n\right\}
\end{gather*}
are such that $(N_{\tau_n\wedge t})_{t\geq 0}$ is a martingale and $|Af_u(B_s)|\leq n$ for $s\in [0,\tau_n)$. In particular, $t\mapsto \int_0^{t\wedge\tau_n} Af_u(B_r)\,dr$ is a continuous process which is of bounded variation on compact $t$-intervals.

The martingale $(N_{t\wedge\tau_n})_{t \geq 0}$ has constant expectation, i.e.\
\begin{gather}\label{main-e08}
    \Ee f_{u}(B_{t\wedge\tau_n}) = f_{u}(0) + \Ee \int_0^{t\wedge\tau_n} Af_{u}(B_r)\,dr.
\end{gather}
By assumption, $|f(x)|\leq C(1+|x|)^N$ for some $N\in\nat$. It is easy to see that $f_u$ is also polynomially bounded and that
\begin{align*}
    |Af_u(x)|
    &= \left|\int_\real A_x p_u(y-x) f(y)\,dy\right|
    \leq \int_\real |Ap_u(y)| \cdot |f(y+x)|\,dy\\
    &\leq c(1+|x|)^N \int_\real |p_u''(y)| (1+|y|)^N\,dy
    \leq c_u(1+|x|)^N.
\end{align*}
Using that the maximum of Brownian motion has moments of any order, we find that
\begin{gather*}
    \Ee \int_0^t |Af_{u}(B_r)|\,dr
    \leq t c_u  \Ee \left[\sup_{r\leq t}(1+|B_r|)^N\right] < \infty.
\end{gather*}
We can apply dominated convergence in \eqref{main-e08} to see that
\begin{gather*}
    h(t) := \Ee f_u(B_{t}) = f_u(0) + \Ee \int_0^t Af_u(B_r)\,dr.
\end{gather*}
This shows, in particular, that $t\mapsto h(t)$ is continuous and of bounded variation. Hence,
\begin{gather*}
    M_{t\wedge\tau_n} - N_{t\wedge\tau_n}
    = \int_0^{t\wedge\tau_n} Af_u(B_r)\,dr - \Ee f_u(B_r)\Big|_{r=t\wedge\tau_n}
    = \int_0^{t\wedge\tau_n} Af_u(B_r)\,dr - h(t\wedge\tau_n)
\end{gather*}
is a continuous martingale which is of bounded variation. Such martingales are a.s.\ constant, cf.\ \cite[Proposition IV.(1.2), p.~120]{rev-yor}, and so
\begin{gather*}
    M_{t\wedge\tau_n} - N_{t\wedge\tau_n}
    = M_{0\wedge\tau_n} - N_{0\wedge\tau_n}
    = - f_u(0).
\end{gather*}
Letting $n\to\infty$ we see that for all $s<t$
\begin{gather*}
    \int_s^t Af_u(B_r)\,dr
    = \lim_{n\to\infty}\left[\int_0^{t\wedge\tau_n} Af_u(B_r)\,dr - \int_0^{s\wedge\tau_n} Af_u(B_r)\,dr\right]
    = \lim_{n\to\infty}\left[h(t\wedge\tau_n)-h(s\wedge\tau_n)\right].
\end{gather*}
The right-hand side is the non-random expression $h(t)-h(s)$. Now we divide by $t-s$, and let $s\uparrow t$. This yields
\begin{gather*}
    Af_u(B_t) = \lim_{s\uparrow t} \frac 1{t-s}\int_s^t Af_u(B_r)\,dr = \text{const}.
\end{gather*}
From this point onwards we can argue as in the first proof, and see that $f_u(x) = a_u x^2 + b_ux + c_u$ for suitable constants $a_u,b_u,c_u\in\real$.

Let us return to the original problem. We need to show that $f_u$ converges in a suitable sense to $f$. First we prove that $f_u(B_t)$ converges to $f(B_t)$ in $L^2(\mbb{P})$. To this end, we note that
\begin{equation*}
	x \mapsto I(\phi,x)
	:= \int_{\mbb{R}} |\phi(x+y)-\phi(y)|^2 p_t(y) \, dy
	= \mbb{E}(|\phi(x+B_t)-\phi(B_t)|^2)
\end{equation*}
is continuous for any continuous function $\phi$, which grows at most polynomially. Since $f$ is, by assumption, at most of polynomial growth, and since the bounded continuous functions $C_b(\real)$ are dense in $L^2(e^{-y^2/4t} \,dy)$, there is a sequence $(\phi_n)_{n \in \mbb{N}}\subset C_b(\real)$ such that $\phi_n \to f$ in $L^2(e^{-y^2/4t}\,dy)$ as $n \to \infty$. We have
\begin{align*}
	|\sqrt{I(\phi_n,x)}-\sqrt{I(f,x)}|^2
	&\leq |\sqrt{I(\phi_n-f,x)}|^2\\
	&\leq 2 \mbb{E}(|\phi_n(x+B_t)-f(x+B_t)|^2) + 2 \mbb{E}(|\phi_n(B_t)-f(B_t)|^2) \\
	&\leq \frac{2}{\sqrt{2\pi t}} \left(1+ e^{x^2/2t}\right)  \int_{\real} |\phi_n(y)-f(y)|^2 e^{-y^2/4t} \, dy,
\end{align*}
where we use in the last step the estimate $2|x| |y| = 2 (\sqrt 2|x|)\, (|y|/\sqrt 2) \leq 2 x^2 + y^2/2$ to obtain
\begin{equation*}
	\exp \left(- \frac{(x-y)^2}{2t} \right) \leq \exp \left( \frac{x^2}{2t} \right) \exp \left(- \frac{y^2}{4t} \right).
\end{equation*}
It follows that $I(\phi_n,x) \to I(f,x)$ uniformly on compact sets, and so $x \mapsto I(f,x)$ is continuous. Thus,
\begin{equation*}
	 \int_{\mbb{R}} I(f,x+z) p_u(x) \, dx
	 \xrightarrow[]{u \to 0} I(f,z)
\end{equation*}
for any $z \in \mbb{R}$. Choosing $z=0$ and plugging in the definition of $I$, we get
\begin{equation*}
	\mbb{E}\left[\mbb{E}'(|f(B_u'+B_t)-f(B_t)|^2)\right]
	=\int_{\mathbb{R}} \mbb{E}(|f(x+B_t)-f(B_t)|^2) p_u(x) \, dx \xrightarrow[]{u \to 0} 0
\end{equation*}
for an independent copy $(B_u')_{u\geq 0}$ of $(B_u)_{u\geq 0}$ and with $\mbb{E}'$ to indicate the expectation acting on $B_u'$. As
\begin{equation*}
	\mbb{E}(|f_u(B_t)-f(B_t)|^2)
	= \mbb{E}(|\mbb{E}'f(B_u'+B_t) - f(B_t)|^2)
	\leq \mbb{E}(\mbb{E}'[|f(B_u'+B_t) - f(B_t)|^2]),
\end{equation*}
we conclude that  $f_u(B_t) \to f(B_t)$ in $L^2(\mbb{P})$. Since the transition density $p_t$ is bounded away from zero on any compact set, this implies
\begin{equation*}
	\int_{|x| \leq k} |f_u(x)-f(x)|^2 \, dx \to 0
\end{equation*}
for any $k \in \mbb{N}$. Consequently, $f_u(x) \to f(x)$ Lebesgue almost everywhere for a suitable subsequence of $f_u$. By the first part of this proof, $f_u$ is a polynomial of order $2$, and so is $f$ (up to a Lebesgue null set).
\end{proof}

The key ingredients in the first proof of Theorem~\ref{main-01} are the following observations:
\begin{itemize}
\item $B_t$ is a L\'evy process,
\item $B_t$ has a transition density which is strictly positive,
\item the generator of $B_t$ satisfies the Liouville theorem,
\end{itemize}
and so the following generalization of Theorem~\ref{main-01} is natural. If $(X_t)_{t\geq 0}$ is a L\'evy process which has a transition density, then the convolution identity $p_{t+s}(x)=p_t*p_s(x)$ shows that $x\mapsto p_t(x)$ is continuous, i.e.\ $p_t>0$ Lebesgue a.e.\ already means that $p_t>0$ on $\real$. Moreover, if $(t,x)\mapsto p_t(x)$ is jointly continuous on $(0,\infty)\times\real$, then $\{x : p_t(x)>0\}$ is either $\real$ or a half-line, i.e.\ $\supp(p_t) = \real$ implies $\{p_t>0\}=\real$; joint continuity is, e.g., implied by $e^{-t\psi(\xi)}\in L^1(d\xi)$ for all $t>0$, see Sharpe \cite{sharpe69} for a discussion, and \cite{kno-schi13} for further results on the existence and smoothness of L\'evy densities.

\begin{theorem}\label{main-05}
    Let $(X_t)_{t \geq 0}$ be a one-dimensional L\'evy process which has a transition density $p_t$ for all $t>0$, such that $p_t>0$ for some $t>0$. Assume that
    \begin{enumerate}
    \item[\itshape either]\textup{(a)}\ \ the characteristic exponent $\psi$ is of class $C^\infty$, and $n\geq 2$ is arbitrary,
    \item[\itshape or]\textup{(b)}\ \ $p_t\in C_b^1(\real)$, $t>0$, and $\Ee |X_1|^{n+\epsilon}<\infty$ for some $\epsilon>0$ and $n\geq 2$.
    \end{enumerate}
    Let $|f(x)|\leq c(1+|x|^{n})$ be a polynomially bounded, measurable function. If $f(X_t)-\Ee f(X_t)$ is a martingale, then $f = \tilde{f}$ almost everywhere for some continuous function $\tilde{f}$ satisfying $A \tilde{f}=\text{const}$, where $A$ is the \textup{(}weakly extended\textup{)} generator. In fact, $\tilde{f}(x)=ax^2+bx+c$ for suitable constants $a,b,c \in \mbb{R}$.
\end{theorem}

\begin{proof}
    Under the assumption (a), we can literally use the first proof of Theorem~\ref{main-01}. Notice that the existence of the density $p_t$ and the (trivial) fact that $p_t\in L^1(dx)$ allows us to use the Riemann--Lebesgue lemma to conclude that $\lim_{|\xi|\to\infty} e^{-t\psi(\xi)}=0$, hence $\lim_{|\xi|\to\infty}\psi(\xi)=\infty$ and $\{\psi = 0\}=\{0\}$.

    Assume now that (b) holds. As in the first proof of Theorem~\ref{main-01}, we find that
	\begin{equation*}
		T_t f(x)-f(x) = \gamma(t) \quad \text{a.e.}
	\end{equation*}
    Since $f\mapsto T_t f$ is a convolution operator, the function $x\mapsto T_t f(x)$ is continuous, and so there is a continuous function $\tilde{f}$ such that $f=\tilde{f}$ almost everywhere and $T_t \tilde{f}(x) - \tilde{f}(x) = \gamma(t)$ for all $x \in \mbb{R}$. As $T_t \tilde{f}(x+y) - \tilde{f}(x+y) = T_t (\tilde{f}(\cdot+y))(x)-\tilde{f}(\cdot+y)(x)= \gamma(t)$, it follows that the difference $\Delta_y \tilde{f}(x):=\tilde{f}(x+y)-\tilde{f}(x)$ satisfies $T_t(\Delta_y \tilde{f})- \Delta_y \tilde{f} = 0$. Hence, $A(\Delta_y \tilde{f})=0$ weakly, and so the Liouville theorem yields that $\Delta_y \tilde{f}$ is a polynomial. By Lemma~\ref{pre-10}, this implies that $\tilde{f}$ is a polynomial.\footnote{It is interesting to note that the particular form of $\gamma(t)$ does not play any role here. To wit: if $u$ is a continuous function such that $T_t u(x)-u(x)=\gamma(t)$ does not depend on $x$, then $u$ is a polynomial. This allows us to avoid Cauchy's functional equation, which we used in the first proof of Theorem~\ref{main-01}.} Thus, by Lemma~\ref{pre-03}, $\gamma(t) = \mathbb{E}\tilde{f}(X_t)$ is a polynomial. In particular, $\gamma(t)$ is differentiable, which implies
	\begin{equation*}
		A\tilde{f}(x)
		= \lim_{t \to 0} \frac{T_t \tilde{f}(x)-\tilde{f}(x)}{t}
		= \lim_{t \to 0} \frac{\gamma(t)}{t}
		=\lim_{t \to 0} \frac{\gamma(t)-\gamma(0)}{t} = \gamma'(0), \quad x \in \mbb{R}.
	\end{equation*}
	Consequently, Lemma~\ref{pre-09} shows that $\tilde{f}$ is a polynomial of degree $\deg(\tilde{f})\leq 2$.
\end{proof}

We will now turn to the `multiplicative' version of Theorem~\ref{main-01}.

\begin{theorem}\label{main-11}
    Let $(B_t)_{t \geq 0}$ be a one-dimensional Brownian motion, and $g: \mbb{R} \to (0,\infty)$ a Borel measurable function which grows at most exponentially. If $g(B_t)/\Ee g(B_t)$ is a martingale, then $g=\tilde{g}$ almost everywhere for a twice differentiable function $\tilde{g}$ which satisfies
	\begin{equation*}
		\frac{A\tilde{g}}{\tilde{g}} = \text{const.}
	\end{equation*}
    where $A$ is the \textup{(}weakly extended\textup{)} generator. Moreover, all positive solutions are of the form $g(x)= a e^{c x} + b e^{-c x}$ almost everywhere for suitable constants $a,b,c \geq 0$.
\end{theorem}

\begin{proof}
    Set $\gamma(t):= \Ee g(B_t)$ and denote by $T_t g(x) = \Ee^x g(B_t) = \Ee g(B_t+x)$ the semigroup generated by the process $(B_t)_{t\geq 0}$. Without loss of generality, we may assume that $g(0)=1$, hence $\gamma(0)=1$.
    From $|g(x)|\leq c e^{\kappa |x|}$ we see that $\sup_{s\leq t}\Ee |g(B_s)|<\infty$ for any $t>0$; in particular, $\gamma(t)$ is strictly positive and bounded in some neighbourhood $[0,\epsilon_0]$ of $t=0$.

    We are first going to show that $\gamma(t)$ is of the form $\gamma(t)=e^{\alpha t}$ for some $\alpha \in \mbb{R}$. We use again the argument from Theorem~\ref{main-01}. Since $g(B_t)/\gamma(t)$ is a martingale, we have
	\begin{gather*}
		\frac{\Ee \left[g(B_{t+s}) \mid \mc{F}_{s}\right]}{\gamma(t+s)} = \frac{g(B_s)}{\gamma(s)}
	\end{gather*}
    for all $s,t>0$, and so, by the Markov property,
    \begin{gather*}
        \frac{g(B_s)}{\gamma(s)}
        = \frac{\Ee \left[g(B_{t+s}) \mid \mc{F}_{s}\right]}{\gamma(t+s)}
        = \frac{\Ee^{B_s} \left[g(B_{t})\right]}{\gamma(t+s)}
        = \frac{(T_t g)(B_s)}{\gamma(t+s)}.
    \end{gather*}
    We can rearrange this to get
    \begin{gather*}
        \frac{(T_tg)(B_s)}{g(B_s)} = \frac{\gamma(t+s)}{\gamma(s)} \quad\text{almost surely},
    \end{gather*}
    which implies, by the strict positivity of the transition density $p_s$ of $B_s$,
    \begin{gather}\label{main-e10}
        \frac{T_tg(x)}{g(x)} = \frac{\gamma(t+s)}{\gamma(s)} \quad\text{Lebesgue almost everywhere}.
    \end{gather}
    Taking $s=0$ and $x=0$, we conclude that
    \begin{gather*}
        \gamma(t+s)=\gamma(t)\gamma(s),\quad s,t\geq 0.
    \end{gather*}
    This means that $t\mapsto\log\gamma(t)$ satisfies the Cauchy functional equation; moreover, it is bounded on the interval $[0,\epsilon_0]$. Therefore (see the proof of Theorem~\ref{main-01}), $t\mapsto\log\gamma(t)$ is linear, and we get $\gamma(t)=\gamma(0) e^{\alpha t} = e^{\alpha t}$ for some $\alpha\in\real$. Taking $s=0$ in \eqref{main-e10}, we obtain that
    \begin{gather*}
        T_t g(x) = e^{\alpha t}g(x)\quad\text{Lebesgue almost everywhere}.
    \end{gather*}
    Since $p_t(y) = (2\pi t)^{-1/2}\exp\left[-\frac 1{2t} y^2\right]$ is a smooth kernel, the convolution $x\mapsto T_t g(x)$ is twice continuously differentiable, and so we find that $g=\tilde g$ Lebesgue a.e.\ for a function $\tilde g \in C^2(\mbb{R})$. Finally, we see that
	\begin{align*}
		A\tilde{g}(x)
		= \lim_{t \to 0} \frac{T_t \tilde{g}(x)-\tilde{g}(x)}{t}
		&= \lim_{t \to 0} \frac{e^{\alpha t}-1}{t} \tilde{g}(x)
		= \alpha \tilde{g}(x).
	\end{align*}

    Since $\tilde{g}$ is twice differentiable, $A\tilde{g}$ has a classical meaning and all solutions of $A\tilde{g}=-\alpha\tilde{g}$ are solutions of the following second-order linear ODE with constant coefficients: $y'' + \alpha y = 0$ and $y(0)=1$. Solving this, we see that $y(x) = a e^{cx}+be^{-cx}$ with $c=\sqrt\alpha\in\mathds{C}$ and $a+b=1$, so all positive solutions are given by $ae^{cx}+be^{-cx}$ with $a,b,c\geq 0$.
\end{proof}

The next result extends Theorem~\ref{main-11} to the class of L\'evy processes with finite exponential moments.
\begin{theorem} \label{main-13}
    Let $(X_t)_{t \geq 0}$ be a  one-dimensional L\'evy process which has a transition density $p_t$ for all $t>0$. Let $g: \mbb{R} \to (0,\infty)$ be a Borel measurable function which grows at most exponentially, i.e. $|g(x)| \leq c e^{\kappa |x|}$. Assume that $\mbb{E}e^{\kappa |X_t|}<\infty$ and
    \begin{enumerate}
        \item[\itshape either]\textup{(a)}\ \ $p_t>0$ on $\mbb{R}$ for some $t>0$
       	    \item[\itshape or]\textup{(b)}\ \ $p_t=0$ on $(-\infty,0)$ and $p_t>0$ on $(0,\infty)$ for some $t>0$.
    \end{enumerate}
    If $g(B_t)/\mbb{E}g(B_t)$ is a martingale, then $g=\tilde{g}$ almost everywhere for a continuous function $\tilde{g}$ satisfying
	\begin{equation*}
		\frac{A \tilde{g}(x)}{\tilde{g}(x)} = \alpha,
	\end{equation*}
    for some constant $\alpha\in\real$, where $A$ denotes the \textup{(}weakly extended\textup{)} generator.

    The equation $\mbb{E}\exp(\lambda X_t)=\exp(\alpha t)$ admits at most two \textup{(}in case \textup{(a)}: at most one\textup{)} solutions $\lambda = \lambda_1,\lambda_2\in\real$. If there are solutions, then $g(x) = a e^{\lambda_1 x} + b e^{\lambda_2 x}$ almost everywhere for suitable constants $a,b \geq 0$; otherwise there is only the trivial solution $g \equiv 0$.
\end{theorem}

\begin{proof}
    Denote by $(T_t)_{t \geq 0}$ the semigroup associated with $(X_t)_{t \geq 0}$, and assume without loss of generality that $g(0)=1$. Following the reasoning from the proof of Theorem~\ref{main-11}, we find that there is a constant $\alpha \in \mbb{R}$ such that $T_t g(x) = e^{\alpha t} g(x)$ for Lebesgue almost all $x$. Since $g\mapsto T_t g$ is a convolution operator, the function $x\mapsto T_t g(x)$ is continuous, and so $g(x) = \tilde{g}(x)$ Lebesgue a.e.\ for a continuous function $\tilde{g}$; in particular,
	\begin{equation}\label{main-e12}
		T_t \tilde{g}(x) = e^{\alpha t} \tilde{g}(x), \quad x \in \mbb{R},
	\end{equation}
    and $A\tilde{g} = \alpha \tilde{g}$. Pick $t>0$ such that the transition density $p_t$ satisfies assumption (a) or (b). Then \eqref{main-e12} can be seen as a convolution equality $\int \tilde{g}(x+y) \, \varrho_t(dy)=\tilde{g}(x)$ with the measure $\varrho_t(dy) = e^{-\alpha t} p_t(y) \, dy$ which has full support in $\mbb{R}$ under the assumption (a), resp.\ in $(0,\infty)$ under the assumption (b).

    Clearly, the equation $\int_{\mbb{R}} e^{\lambda x} \, \varrho_t(dx)=1$ can be rewritten in the form $M(\lambda):=\mbb{E}e^{\lambda X_t} = e^{\alpha t}$. Since we have assumed that $\Ee e^{\kappa |X_t|}<\infty$, it is easy to see that $M''(\lambda) = \Ee\left[ X_t^2 e^{\lambda X_t}\right] > 0$ exists and is strictly positive for all $\lambda\in (-\kappa,\kappa)$. Thus, $M$ is strictly convex on $[-\kappa,\kappa]$, and from this it is clear that $M(\lambda) = e^{\alpha t}$ has one or two solutions if $\min\left\{\Ee e^{-\kappa X_t}, 1, \Ee e^{+\kappa X_t}\right\} \leq e^{\alpha t}\leq \max\left\{\Ee e^{-\kappa X_t}, \Ee e^{+\kappa X_t}\right\}$.\footnote{If $X_t$ has exponential moments of all orders, we can take $\kappa\to\infty$ and we see, letting $\lambda\to\pm\infty$ in
    \begin{gather*}
        \Ee e^{\lambda X_t} = \Ee\left[e^{\lambda X_t}\I_{\{X_t>0\}}\right] + \Ee\left[e^{\lambda X_t}\I_{\{X_t<0\}}\right],
    \end{gather*}
    that in the regime (b) the function $M(\lambda)$ is $U$-shaped with the global minimum $M(0)=1$ at $\lambda = 0$; thus there are always two solutions if $e^{\alpha t}>1$, one solution if $e^{\alpha t}=1$, and no solution otherwise. In the regime (a), $M(\lambda)$ has the limits $0$ and $\infty$ as $\lambda\to\pm\infty$, and there is always exactly one solution.}

    If there are solutions $\lambda_1\neq \lambda_2$ or $\lambda_1=\lambda_2$, then all positive locally integrable solutions satisfying \eqref{main-e12} are of the form
	\begin{equation*}
		\tilde{g}(x) = a e^{\lambda_1 x} + b e^{\lambda_2 x}, \quad x \in \mbb{R},
	\end{equation*}
    for $a,b \geq 0$; otherwise, \eqref{main-e12} has only the trivial solution $\tilde{g}=0$; see the classical works by Choquet \& Deny \cite{choquet60, deny} or the elementary presentation by Ramachandran \& Prakasa Rao \cite{ramachandran} for the case that $\varrho_t$ has full support in $\mbb{R}$ and Lau \& Rao \cite{lau-rao} for the case that $\varrho_t$ is supported in $[0,\infty)$. From the definition of $\varrho_t$, it is immediate that the condition on $\lambda_i$ is equivalent to  $\mbb{E}e^{\lambda_i X_t} = e^{\alpha t}$.
\end{proof}

\begin{ack}
    We are grateful to our colleague David Berger for suggesting to consider the difference $\Delta_y$ in the proof of Theorem~\ref{main-05}; this helped us to remove the condition $\mathbb{E}(X_t)=0$ in an earlier draft of this note. We would like to thank Nick Bingham, Charles Goldie and the anonymous referee for their valuable comments, which improved the presentation of this paper. We are grateful to Lucian Beznea from Bukarest for his help to get a copy of the paper \cite{angel}. Financial support through the joint Polish--German NCN--DFG `Beethoven 3' grant (NCN 2018/31/G/ST1/02252; DFG SCHI 419/11-1) is gratefully acknowledged.
\end{ack}

\end{document}